\documentclass[11pt,dvips,twoside,letterpaper]{article}
\usepackage{pslatex}
\usepackage{fancyhdr}
\usepackage{graphicx}
\usepackage{geometry}

\def\figurename{Figure} 
\makeatletter
\renewcommand{\fnum@figure}[1]{\figurename~\thefigure.}
\makeatother

\def\tablename{Table} 
\makeatletter
\renewcommand{\fnum@table}[1]{\tablename~\thetable.}
\makeatother

\usepackage{amsmath}
\usepackage{amssymb}
\usepackage{amsfonts}
\usepackage{amsthm,amscd}

\newtheorem{theorem}{Theorem}[section]
\newtheorem{lemma}[theorem]{Lemma}

\newtheorem{proposition}[theorem]{Proposition}
\theoremstyle{definition}
\newtheorem{definition}[theorem]{Definition}
\newtheorem{example}[theorem]{Example}

\theoremstyle{remark}
\newtheorem{remark}[theorem]{Remark}

\numberwithin{equation}{section}

\def\P{\mathbb P}
\def\R{\mathbb R}
\def\E{\mathbb E}

\def\E{\mathbb E}
\def\N{\mathbb N}
\def\cal{\mathcal}

\begin{document}
\title{\bfseries\scshape{Reflected generalized backward doubly SDEs driven by L\'{e}vy processes and Applications}}
\author{\bfseries\scshape Auguste Aman\thanks{augusteaman5@yahoo.fr,\ Corresponding author.}\\
U.F.R.M.I, Universit\'{e} de Cocody, \\582 Abidjan 22, C\^{o}te d'Ivoire}

\date{}
\maketitle \thispagestyle{empty} \setcounter{page}{1}

\begin{abstract}
In this paper, a class of reflected generalized backward doubly
stochastic differential equations (reflected GBDSDEs in short) driven by Teugels martingales
associated with L\'{e}vy process and the integral with respect to an
adapted continuous increasing process is investigated. We obtain the existence
and uniqueness of solutions to these equations. A probabilistic interpretation
for solutions to a class of reflected stochastic partial differential integral
equations (PDIEs in short) with a nonlinear Neumann boundary condition
is given.
\end{abstract}


\vspace{.08in} \noindent \textbf{Keywords}: Reflected backward doubly SDEs, stochastic
partial differential integral equation; L\'{e}vy process; Teugels martingale;
Neumann boundary condition.

\section{Introduction}
Backward stochastic differential equations (BSDEs, in short) have been first
introduced by Pardoux and Peng \cite{PP1} in order to give a probabilistic interpretation
(Feynman-Kac formula) for the solutions of semilinear parabolic
PDEs, one can see Peng \cite{P}, Pardoux and Peng \cite{PP2}. Recently, a new class
of BSDEs, named backward doubly stochastic differential equations (BDSDEs
in short) has been introduced by Pardoux and Peng \cite{PP3} in order to
give a probabilistic representation for a class of quasilinear stochastic partial
differential equations (SPDEs in short). Following it, Bally and Matoussi
\cite{BM} gave the probabilistic representation of the weak solutions to parabolic
semilinear SPDEs in Sobolev spaces by means of BDSDEs. Furthermore,
Pardoux and Zhang \cite{PZ} gave a probabilistic formula for the viscosity solution
of a system of PDEs with a nonlinear Neumann boundary condition by introducing a generalized BSDEs (GBSDEs, in short) which involved an
integral with respect to an adapted continuous increasing process. Its extension
to an obstacle problem for PDEs with a nonlinear Neumann boundary
condition was given in Ren and Xia \cite{RX} by reflected GBSDEs. Motivated
by the above works, especially by \cite{PP3} and \cite{PZ}, Boufoussi et al. \cite{Boual1} recommended
a class of generalized BDSDEs (GBDSDEs in short) and gave the
probabilistic representation for stochastic viscosity solutions of semi-linear
SPDEs with a Neumann boundary condition.
The main tool in the theory of BSDEs is the martingale representation
theorem, which is well known for martingale which adapted to the filtration
of the Brownian motion or that of Poisson point process (Pardoux and Peng
\cite{PP1}, Tang and Li \cite{TL}) or that of a Poisson random measure ( see Ouknine \cite{O}).
Recently, Nualart and Schoutens \cite{NS1} gave a martingale representation theorem
associated to L\'{e}vy process. Furthermore, they showed the existence
and uniqueness of solutions to BSDEs driven by Teugels martingales associated
with L\'{e}vy process with moments of all orders in \cite{NS2}. The results were
important from a pure mathematical point of view as well as in the world of
finance. It could be used for the purpose of option pricing in a L\'{e}vy market
and related PDEs which provided an analogue of the famous Black-Scholes
formula. Further, Hu and Yong considered respectively BDSDEs and generalized BDSDE driven by L\'{e}vy processes
and its applications in \cite{HY1} and \cite{HY2}.

Motivated by the above works, especially by \cite{HY2} the purpose of the present paper is to
consider reflected GBDSDEs driven by L\'{e}vy processes of the kind considered in Nualart
and Schoutens \cite{NS1}. Our aim is to give a probabilistic interpretation
for the solutions to a class of reflected stochastic PDIEs with a nonlinear Neumann
boundary condition.

The paper is organized as follows. In Section 2, we introduce some
preliminaries and notations. Section 3 is devoted to GBDSDEs driven by L\'{e}vy processes and the comparison theorem related to it. In
Section 4, we give existence and uniqueness result for the reflected GBDSDE. Finally Section 5 point out a probabilistic interpretation of solutions to a class of
reflected stochastic PDIEs with a nonlinear Neumann boundary condition.

\section{Preliminaries and Notations}
\setcounter{theorem}{0} \setcounter{equation}{0}
The scalar product of the space $\R^{d} (d\geq 2)$ will be denoted
by $<.>$ and the associated Euclidian norm  by $\|.\|$.

In what follows let us fix a positive real number
$T>0$. Let $(\Omega, \mathcal{F},\P,\mathcal{F}_{t},B_{t}, L_{t}: t\in [0, T])$ be a complete Wiener-L\'{e}vy space in $\R \times\R\backslash\{0\}$, with Levy measure $\nu$, i.e. $(\Omega, \mathcal{F},\P)$ is a complete probability space, $\{\mathcal{F}_{t}: t\in[0, T]\}$ is a right-continuous increasing family of complete sub $\sigma$-algebras of $\mathcal{F}$, $\{B_t: t\in [0, T]\}$ is a standard Wiener process in $\R$ with respect to $\{\mathcal{F}_{t}: t\in[0, T]\}$ and $\{L_{t}: t\in [0, T]\}$ is a $\R$-valued L\'{e}vy process independent of $\{B_{t} : t \in[0, T]\}$, which has only $m$ jumps size and no continuous part and corresponding to a standard L\'{e}vy measure $\nu$ satisfying the following conditions:
$
\begin{array}{l}
\int_{\R}(1 \wedge y)\nu(dy) < \infty,
\end{array}
$

Let $\mathcal{N}$ denote the totality of $\P$-null sets of $\mathcal{F}$. For each $t\in[0, T]$, we define that
\begin{eqnarray*}
\mathcal{F}_{t}= \mathcal{F}_{t}^{L}\vee \mathcal{F}_{t,T}^{B}
\end{eqnarray*}
where for any process $\{\eta_t\},\; \mathcal{F}_{s,t}^{\eta}=\sigma(\eta_{r}-\eta_s, s\leq r \leq t)\vee
\mathcal{N},\; \mathcal{F}_{t}^{\eta}=\mathcal{F}_{0,t}^{\eta}$.

Let us remark that the collection ${\bf F}= \{\mathcal{F}_{t},\ t\in
[0,T]\}$ is neither increasing nor decreasing and it does not
constitute a filtration.

We denote by $(H^{(i)})_{i\geq 1}$ the Teugels Martingale associated with the L\'{e}vy
process $\{L_t : t\in[0, T]\}$. More precisely
\begin{eqnarray*}
H^{(i)}=c_{i,i}Y^{(i)}+c_{i,i-1}Y^{(i-1)}+\cdot\cdot\cdot+c_{i,1}Y^{(1)}
\end{eqnarray*}
where $Y^{(i)}_{t}=L_{t}^{i}-\E(L^{i}_t)=L_{t}^{i}-t\E(L^{1}_t)$ for all $i\geq 1$ and $L^{i}_t$ are power-jump processes. That is $L^{1}_t=L_t$ and $L^{i}_t=\sum_{0<s<t}(\Delta L_s)^i$ for all $i\geq 2$, where $X_{t^-} = \lim_{s\nearrow t} X_s$ and $\Delta X_t = X_t - X_{t^-}$. It was shown in Nualart and Schoutens \cite{NS1} that the coefficients $c_{i,k}$ correspond to the orthonormalization of the polynomials $1, x, x^2,...$ with respect to the measure $\mu(dx)=x^{2}d\nu(x)+\sigma^{2}\delta_{0}(dx)$:
\begin{eqnarray*}
q_{i-1}(x)=c_{i,i}x^{i-1}+c_{i,i-1}x^{i-2}+\cdot\cdot\cdot+c_{i,1}.
\end{eqnarray*}
We set
\begin{eqnarray*}
p_i(x)=xq_{i-1}(x)=c_{i,i}x^{i}+c_{i,i-1}x^{i-1}+\cdot\cdot\cdot+c_{i,1}x^{1}.
\end{eqnarray*}

The martingale $(H^{(i)})_{i\geq 1}$ can be chosen to be pairwise strongly orthonormal
martingale.
\begin{remark}
\begin{enumerate}
\item If $\mu$ only has mass at $1$, we are in the Poisson case; here
$H^{(i)}_t= 0, i = 2,\cdot\cdot\cdot$. This case is degenerate in this L\'{e}vy framework
\item Generally, if the L\'{e}vy process $L$ has only $m$ different jump sizes, then
\begin{description}
\item $(i)\, H^{(k)} = 0, \forall\; k\geq m+1$, if $L$ has no continuous part;
\item $(ii)\, H^{(k)} = 0, \forall\; k\geq m+2$, if $L$ has continuous part.
\end{description}
\end{enumerate}
\end{remark}

In the sequel, let\
$\{A_t,\ 0\leq t\leq T\}$\ be a continuous, increasing and
${\bf F}$-adapted real valued with bounded
variation on $[0,T]$ such that\ $A_0=0$.

For any $d\geq 1$, we consider the following spaces of processes:
\begin{enumerate}
\item $\mathcal{M}^{2}(\R^{d})$ denote the space of real valued, square integrable and $\mathcal{F}_{t}$-predictable processes $\varphi=\{\varphi_{t};
t\in[0,T]\}$ such that
\begin{description}
\item $\|\varphi\|^{2}_{{\mathcal{M}}^{2}}=\E\int ^{T}_{0}\|\varphi_{t}\|^{2}dt<\infty$.
\end{description}
\item  $\mathcal{S}^{2}(\R)$ is the subspace of  $\mathcal{M}^{2}(\R)$ formed by the $\mathcal{F}_{t}$-adapted processes $\varphi=\{\varphi_{t};
t\in[0,T]\}$ right continuous with left limit (rcll) such that
\begin{description}
\item $\displaystyle{\|\varphi\|^{2}_{\mathcal{S}^{2}}=\E\left(\sup_{0\leq t\leq T}|\varphi_{t}|
^{2}+\int ^{T}_{0}|\varphi_{t}|^{2}dA_t\right)<\infty}$.
\end{description}
\item  $\mathcal{A}^{2}(\R)$ is the set of $\mathcal{F}_{t}$-measurable, continuous, real-valued,
increasing process $\varphi=\{\varphi_{t}; t\in[0,T]\}$ such that $K_0 = 0,\; \E|K_T|^2 < \infty$
\end{enumerate}
Finally we denote by $\mathcal{E}^{2,m}=\mathcal{S}^{2}(\R)\times{\mathcal{M}}^{2}(\R^{m})\times\mathcal{A}^{2}(\R)$ endowed with the norm
\begin{eqnarray*}
\|(Y,Z,K)\|^{2}_{\mathcal{E}}=\E\left(\sup_{0\leq t\leq T}|Y_{t}|
^{2}+\int ^{T}_{0}|Y_{t}|^{2}dA_t+\int ^{T}_{0}\|Z_{t}\|^{2}dt+|K_{T}|^{2}\right).
\end{eqnarray*}
Then, the couple $(\mathcal{E}^{2,m}, \|.\|_{\mathcal{E}^{2,m}})$ is a Banach space.

To end this section, let us give following needed assumptions
\begin{itemize}
\item[$(\textbf{H1})$] $\xi$ is a square integrable random variable which is $\mathcal{F}_{T}$-measurable such  that for all $\mu>0$
$$ \mathbb{E}\left(e^{\mu A_T}|\xi|^2\right) < \infty. $$
\item[$(\textbf{H2})$]
$f:\Omega\times [0,T]\times\R\times \R^{m}\rightarrow \R$,\
$g:\Omega\times [0,T]\times\R\rightarrow \R,$\
and $\phi:\Omega\times [0,T]\times\R \rightarrow \R,$ are three
functions such that:
\begin{itemize}
  \item[$(a)$] There exist $\mathcal{F}_t$-adapted processes $\{f_t,\,
\phi_t,\,g_t:\,0\leq t\leq T\}$ with values in $[1,+\infty)$ and
with the property that for any $(t,y,z)\in
[0,T]\times\R\times\R^{d}$, and $\mu>0$, the following hypotheses
are satisfied for some strictly positive finite constant $K$:
\begin{eqnarray*}
\left\{
\begin{array}{l}
f(t,y,z),\, \phi(t,y),\,\mbox{and}\, g(t,y,z)\, \mbox{are}\, \mathcal{F}_t\mbox{-measurable processes},\\\\
|f(t,y,z)|\leq f_t+K(|y|+\|z\|),\\\\
|\phi(t,y)|\leq \phi_t+K|y|,\\\\
|g(t,y)|\leq g_t+K|y|,\\\\
\displaystyle \E\left(\int^{T}_{0} e^{\mu
A_t}f_t^{2}dt+\int^{T}_{0}e^{\mu A_t}g_t^{2}dt+\int^{T}_{0}e^{\mu
A_t}\phi_t^{2}dA_t\right)<\infty.
\end{array}\right.
\end{eqnarray*}
\item[$(b)$] There exist constants $c>0, \beta<0$ and $0<\alpha<1$ such that for any $(y_1,z_1),\,(y_2,z_2)\in\R\times\R^{m}$,
\begin{eqnarray*}
\left\{
\begin{array}{l}
(i)\, |f(t,y_1,z_1)-f(t,y_2,z_2)|^{2}\leq c(|y_1-y_2|^{2}+\|z_1-z_2\|^{2}),\\\\
(ii)\, |g(t,y_1)-g(t,y_2)|^{2}\leq c|y_1-y_2|^{2},\\\\
(iii)\; \langle y_1-y_2,\phi(t,y_1)-\phi(t,y_2)\rangle\leq
\beta|y_1-y_2|^{2}.
\end{array}\right.
\end{eqnarray*}
\end{itemize}
\item[($\textbf{H3}$)] The obstacle $\left\{  S_{t},0\leq t\leq T\right\}  $,
is a $\mathcal{F}_{t}$-progressively measurable
real-valued process satisfying $$E\left(  \sup_{0\leq t\leq T}\left|
S^+_{t}\right| ^{2}\right)  <\infty.$$ We shall always assume that
$S_{T}\leq\xi\ a.s.$
\end{itemize}

\section{Generalized backward doubly stochastic differential equations driven by L\'{e}vy processes}
In this section, we present existence and uniqueness results for GBDSDEs driven by L\'{e}vy processes and we prove a comparison theorem
which is an important tool in the proofs for results of Sections 4. The existence and uniqueness result is a direct consequence of Theorem 3.2 in \cite{HY1}.
\begin{proposition}\label{HY2}
Given standard parameter $(\xi, f,\phi, g)$, there exists $(Y,Z)\in{\mathcal{S}}^{2}(\R)\times{\mathcal{M}}^{2}(\R^{m})$
to the following GBDSDEs driven by the L\'{e}vy processes
\begin{eqnarray}
Y_{t}&=&\xi+\int_{t}^{T}f(s,Y_{s^-},Z_{s})ds+\int_{t
}^{T}\phi(s,Y_{s^-})dA_s+\int_{t}^{T}g(s,Y_{s^-})\,dB_{s}\nonumber\\
&&-\sum_{i=1}^{m}\int_{t}^{T}Z^{(i)}_{s}dH^{(i)}_{s},\,\ 0\leq t\leq
T.\label{a0}
\end{eqnarray}
Here the integral with respect to $\{B_t\}$ is the classical backward It\^{o} integral
(see Kunita \cite{K}) and the integral with respect to $\{H^{(i)}_t\}$ is a standard forward
It\^{o}-type semimartingale integral.
\end{proposition}

The comparison theorem is one of the principal tools in the theories of the BSDEs. But it does not hold in general for
solutions of BSDEs with jumps (see the counter-example in Barles et al. \cite{Bal}). In the following we prove, with the
additional property of the jumps size, the comparison theorem for solution of GBDSDEs driven by L\'{e}vy processes. Let note that in the standard BSDE case i.e $g=\phi=0$, comparison theorem has already been established by Qing Zhou \cite{QZ} with this property of jumps size.
\begin{theorem} \label{Thm:comp}
Assume that $L$ has only $n$ different jump sizes and has no continuous part. Let $(\xi^1,f^1,\phi,g)$ and $(\xi^2,f^2,\phi,g)$ be two standard parameters of BSDE $(\ref{a0})$ and let $(Y^1,Z^1)$ and $(Y^2,Z^2)$ be the associated square-integrable solutions. Suppose that
\begin{enumerate}
\item $\xi^1\geq\xi^2,\; \P \;$ a.s.,
\item $f^1(t,y,z)\geq f^2(t,y,z),\; \P\;$ a.s. for all $y\in\R,\; z\in\R^m$,
\item $\displaystyle{\beta^{i}_t=\frac{f(t,Y^2_{t^-},\tilde{Z}_{t}^{(i-1)})-f^2(t,Y^2_{t^-},\tilde{Z}_{t}^{(i)})}{Z^{1(i)}_t-Z^{2(i)}_t}{\bf 1}_{\{Z^{1(i)}_t-Z^{2(i)}_t\neq 0\}}}$
\end{enumerate}
where
\begin{eqnarray*}
\tilde{Z}^{i}&=&\left(Z^{2(1)},Z^{2(2)},\cdot\cdot\cdot,Z^{2(i)},Z^{1(i+1)},\cdot\cdot\cdot,Z^{1(n)}\right)\\
\tilde{Z}^{i-1}&=&\left(Z^{2(1)},Z^{2(2)},\cdot\cdot\cdot,Z^{2(i-1)},Z^{1(i)}, Z^{1(i+1)},\cdot\cdot\cdot,Z^{1(n)}\right),
\end{eqnarray*}
satisfying that $\displaystyle{\sum_{i=1}^{m}\beta^{i}_t\Delta H^{i}_t>-1,\; dt\otimes d\P}$ a.s.
Then we have that almost surely for any time $t,\; Y^1_t\geq Y^2_t$ and that if $\P(\xi^1>\xi^2)>0$ then $\P(Y^1_t>Y^2_t)>0$.
\end{theorem}
\begin{proof}
Denote
\begin{eqnarray*}
\hat{\xi} &=& \xi^{1}-\xi^{2},\;\;\;\; \hat{Y}_t = Y_t^{1}-Y_t^{2},\;\;\;\;\; \hat{V}_t = V^{1}_t-V^{2}_t,\;\;\;\;\;\; \hat{Z}_t = Z^{1}_t-Z^{2}_t\\
\hat{f}_t &=&f^1(t, Y^ 2_t , Z^2_t )-f^{2}(t, Y^{2}_t, Z^{2}_t),
\end{eqnarray*}
and
\begin{eqnarray*}
a_t &=& [f^1(t, Y^1_t,Z^1_t) - f^1(t, Y^2_t,Z^1_t)]/(Y^1_t-Y^2_t){\bf 1}_{\{Y^1_t\neq Y^2_t\}},\\
b_t &=& [\phi(t, Y^1_t)-\phi(t,Y^2_t,)]/(Y^1_t-Y^2_t){\bf 1}_{\{Y^1_t\neq Y^2_t\}},\\
c_t &=& [g(t, Y^1_t) - g(t, Y^2_t)]/(Y^1_t-Y^2_t){\bf 1}_{\{Y^1_t\neq Y^2_t\}}.
\end{eqnarray*}
Then
\begin{eqnarray*}
\hat{Y}_t &=& \hat{\xi}+\int_{t}^{T}[a_s\hat{Y}_{s^-}+\sum_{i=1}^{m}\beta^{i}_s\hat{Z}^{(i)}_s+\hat{f}_s]ds+\int_{t}^{T}b_s\hat{Y}_{s^-}dA_s+\int_{t}^{T}c_s\hat{Y}_{s^-}dB_s
-\sum^{m}_{i=1}\int_t^T\hat{Z}^{(i)}_sdH^{(i)}_s
\end{eqnarray*}
is a linear GBDSDE driven the L\'{e}vy processes.

Let $\displaystyle{\Gamma_t= 1 + \int_{0}^{t}\Gamma_{s^-}dX_s}$, where
\begin{eqnarray*}
 X_t=\int_{0}^{t}a_sds+\int_{0}^{t}b_sdA_s+\int_{0}^{t}c_sdB_s-\int_{0}^{t}|c_s|^{2}ds+\sum_{i=1}^{m}\int_{0}^{t}\beta^{i}_sdH^{(i)}_s.
 \end{eqnarray*}Then we have $\displaystyle{\Delta X_t =
\sum_{i=1}^{m}\beta^{i}_t\Delta H^{(i)}_t >-1}$. Note that $|a_t |\leq C,\; |b_t|\leq C,\; |c_t|\leq C,\;  |\beta^{i}_t|\leq C$, for all $0\leq t\leq T$, a.s., $i=1,\cdot\cdot\cdot,m$. Then by the Dol\'{e}ans-Dade exponential formula and the Gronwall inequality, we conclude that $\Gamma_t>0$ and $\sup_{0\leq t\leq T}\E[\Gamma^{2}_t]\leq C_1$. Thus, $\E[\int_{0}^{T}\Gamma^{2}_{s^-}ds]\leq C_1$, where $C_1$ is a positive constant. Then applying It\^{o}'s formula to
$\Gamma_s\hat{Y}_s$ from $s = t$ to $s = T$, it follows that
\begin{eqnarray}
\Gamma_T\hat{\xi}-\Gamma_t\hat{Y}_t&=&\int^{T}_{t}\Gamma_{s^-}\hat{Y}_{s^-}d\hat{Y}_s+\int^{T}_{t}\hat{Y}_{s^-}\Gamma_{s^-}d\Gamma_s+\int^{T}_{t}d[\Gamma,\hat{Y}]_s\nonumber\\
&=& -\int_{t}^{T}\Gamma_{s^-}[a_s\hat{Y}_{s^-}+\sum_{i=1}^{m}\beta^{i}_s\hat{Z}^{(i)}_s+\hat{f}_s]ds
-\int_{t}^{T}\Gamma_{s^-}b_s\hat{Y}_{s^-}dA_s-\int_{t}^{T}\Gamma_{s^-}\hat{Y}_{s^-}c_sdB_s\nonumber\\
&&+\sum^{m}_{i=1}\int_t^T\Gamma_{s^-}\hat{Z}^{(i)}_sdH^{(i)}_s+\int_{t}^{T}a_s\hat{Y}_{s^-}\Gamma_{s^-}ds+\int_{t}^{T}\hat{Y}_{s^-}\Gamma_{s^-}b_sdA_s
+\int_{t}^{T}\hat{Y}_{s^-}\Gamma_{s^-}c_sdB_s\nonumber\\
&&+\sum_{i=1}^{m}\int_{t}^{T}\hat{Y}_{s^-}\Gamma_{s^-}\beta^{i}_sdH^{(i)}_s
+\sum_{i=1}^{m}\sum_{j=1}^{m}\int_{t}^{T}\Gamma_{s^-}\beta^{i}_s\hat{Z}_{s}^{(j)}d[H^{(i)},H^{(j)}]_s\nonumber\\
&=& -\int_{t}^{T}\Gamma_{s^-}[\sum_{i=1}^{m}\beta^{i}_s\hat{Z}^{(i)}_s+\hat{f}_s]ds
+\sum^{m}_{i=1}\int_t^T\Gamma_{s^-}\hat{Z}^{(i)}_sdH^{(i)}_s+\sum_{i=1}^{m}\int_{t}^{T}\hat{Y}_{s^-}\Gamma_{s^-}\beta^{i}_sdH^{(i)}_s\nonumber\\
&&+\sum_{i=1}^{m}\sum_{j=1}^{m}\int_{t}^{T}\Gamma_{s^-}\beta^{i}_s\hat{Z}_{s}^{(j)}d[H^{(i)},H^{(j)}]_s.\label{Girsanov}
\end{eqnarray}
By Davis's inequality, we know that $\displaystyle{\sum^{m}_{i=1}\int_t^T\Gamma_{s^-}\hat{Z}^{(i)}_sdH^{(i)}_s}$ and $\displaystyle{\sum_{i=1}^{m}\int_{t}^{T}\hat{Y}_{s^-}\Gamma_{s^-}\beta^{i}_sdH^{(i)}_s}$ are martingales.
Since $\displaystyle{\sum^{m}_{i=1}\int_t^T\hat{Z}^{(i)}_sdH^{(i)}_s}$ and $\displaystyle{\sum_{i=1}^{m}\int_{t}^{T}\Gamma_{s^-}\beta^{i}_sdH^{(i)}_s}$ are square integrable martingales, we have
\begin{eqnarray*}
\E\left[\sum_{i=1}^{m}\sum_{j=1}^{m}\int_{t}^{T}\Gamma_{s^-}\beta^{i}_s\hat{Z}_{s}^{(j)}d[H^{(i)},H^{(j)}]_s|\mathcal{F}_t\right]
&=&\E\left[\sum_{i=1}^{m}\sum_{j=1}^{m}\int_{t}^{T}\Gamma_{s^-}\beta^{i}_s\hat{Z}_{s}^{(j)}d\langle H^{(i)},H^{(j)}\rangle_s|\mathcal{F}_t\right]\\
&=&\E\left[\sum_{i=1}^{m}\int_{t}^{T}\Gamma_{s^-}\beta^{i}_s\hat{Z}_{s}^{(i)}ds|\mathcal{F}_t\right].
\end{eqnarray*}
Thus, from $(\ref{Girsanov})$, and taking conditional expectation w.r.t. $\mathcal{F}_t$, we conclude that
\begin{eqnarray*}
\Gamma_t\hat{Y}_t=\E\left[\Gamma_T\hat{\xi}+\int_{t}^{T}\Gamma_{s^-}\hat{f}_sds|\mathcal{F}_t\right]\geq 0.
\end{eqnarray*}
It is clear that $\hat{Y}_t\geq 0$ and that if $\P(\hat{\xi} > 0) > 0$ then $\P(\hat{Y}_t > 0) > 0$. The proof of the theorem
is complete.
\end{proof}
\section{Reflected generalized backward doubly stochastic differential equation driven by L\'{e}vy processes}
This section is devoted to the study of reflected GBDSDEs driven by the L\'{e}vy processes $(\ref{a1})$, one of our main goal in this paper.  First of all let us give a definition to the solution of this reflected GBDSDEs driven by L\'{e}vy processes.
\begin{definition}
By a solution of the reflected GBDSDE $(\xi,f,\phi,g,S)$
driven by L\'{e}vy processes we mean a triplet of processes  $(Y,Z,K)\in\mathcal{E}$, which satisfied
\begin{eqnarray}
Y_{t}&=&\xi+\int_{t}^{T}f(s,Y_{s^-},Z_{s})ds+\int_{t
}^{T}\phi(s,Y_{s^-})dA_s+\int_{t}^{T}g(s,Y_{s^-})\,dB_{s}\nonumber\\
&&-\sum_{i=1}^{\infty}\int_{t}^{T}Z^{(i)}_{s}dH^{(i)}_{s}+K_{T}-K_t,\,\ 0\leq t\leq
T.\label{a1}
\end{eqnarray}
such that the following holds $\P$-a.s
\begin{description}
\item  $(i)$ $Y_{t}\geq S_{t},\,\,\,\ 0\leq t\leq T$,\,\
\item  $(ii)$\ $\displaystyle \int_{0}^{ T}\left( Y_{t^-}-S_{t}\right)
dK_{t}=0$.
\end{description}
\end{definition}

In the sequel, $C$ denotes a finite constant which may take different values
from line to line and usually is strictly positive.
\begin{theorem}\label{Th:exis-uniq}
Under the hypotheses $({\bf H1})$,\ $({\bf H2})$ and
$({\bf H3})$, there exists a unique solution for the reflected
generalized BDSDE $(\xi,f,\phi,g,S)$ driven by L\'{e}vy processes.
\end{theorem}

Our proof is based on a penalization method from El Karoui et al \cite{Kal}.

For each $n\in\N^{*}$ we set
\begin{eqnarray}\label{a0}
f_{n}(s,y,z)=f(s,y,z)+n(y-S_{s})^{-}
\end{eqnarray}
and let $(Y^{n},Z^{n})$ be the $\mathcal{F}_t$-progressively measurable
process with values in $\R\times \R^{m}$ unique solution of the GBDSDE
with $(\xi,f_n,g)$ driven by the L\'{e}vy processes. It exists according to Proposition $\ref{HY2}$. So
\begin{eqnarray*}
\E\left(\sup_{0\leq t\leq
T}|Y^{n}_{t}|^{2}+\int_0^T\left\|Z_s^{n}\right\|^2 ds
\right)<\infty,
\end{eqnarray*}
and
\begin{eqnarray}
Y_{t}^{n}&=&\xi+\int_{t}^{T}f(s,Y_{s^-}^{n},Z_{s}^{n})ds+\int_{t}^{T}(Y^{n}_{s^-}-S_s)^{-}ds+\int_{t}^{T}\phi(s,Y_{s^-}^{n})dA_s\nonumber\\
&&+\int_{t}^{T}g(s,Y_{s^-}^{n})\,dB_{s}-\sum_{i=1}^{m}\int_{t}^{T}(Z_{s}^{n})^{(i)}dH^{(i)}_{s}. \label{h2}
\end{eqnarray}
Set
\begin{eqnarray}
K_{t
}^{n}=n\int_{0}^{t}(Y_{s^-}^{n}-S_{s})^{-}ds , \,\ 0\leq t\leq
T.\label{K}
\end{eqnarray}
In order to prove Theorem \ref{Th:exis-uniq}, we state the following
lemma that will be useful.
\begin{lemma}\label{lem:1}
Let us consider  $(Y^{n},Z^{n})\in \mathcal{S}^{2}(\R)\times
{\mathcal{M}}^{2}(\R^{m})$ solution of GBDSDE $(\ref{h2})$. Then there exists
$C>0$ such that,
\begin{eqnarray*}
\sup_{n\in\N^{*}}\E\left( \sup_{0\leq t\leq T}\left|
Y_{t}^{n}\right| ^{2}+\int_{t}^{T}\left| Y_{s}^{n}\right| ^{2}
dA_s+\int_{t}^{T}\left\| Z_{s}^{n}\right\| ^{2}
ds+|K^{n}_{T}|^{2}\right)<C
\end{eqnarray*}
\end{lemma}

\begin{proof}
From It\^{o}'s formula, we have
\begin{align}\label{b2}
\left| Y_{t}^{n}\right|^{2}
&=\left| \xi\right|
^{2}+2\int_{t}^{T}Y_{s^-}^{n}f(s,Y_{s^-}^{n},Z_{s}^{n}) ds
+2\int_{t}^{T}Y_{s^-}^{n}\phi(s,Y_{s^-}^{n})dA_s\nonumber\\
&+\int_{t}^{T}|g(s,Y_{s^-}^{n})|^{2}ds +2\int_{t}^{T} Y_{s^-}^n dK_{s}^{n}+ 2\int_{t}^{T}Y_{s^-}^{n}
g(s,Y_{s^-}^{n})dB_{s}\nonumber\\
&-2\sum_{i=1}^{m}\int_{t}^{T}Y_{s^-}^{n}(Z_{s}^{n})^{(i)}dH^{(i)}_{s}-\sum_{i,j=1}^{m}\int_{t}^{T}(Z_{s}^{n})^{(i)}(Z_{s}^{n})^{(j)}d[H^{(i)}_{s},H^{(j)}_{s}].
\end{align}
Note that $\int_{t}^{T}Y_{s^-}^{n}g(s,Y_{s^-}^{n})dB_{s},\;\int_{t}^{T}Y_{s^-}^{n}(Z_{s}^{n})^{(i)}dH^{(i)}_{s}$, for $i\geq 1$ and $
\int_{t}^{T}(Z_{s}^{n})^{(i)}(Z_{s}^{n})^{(j)}d[H^{(i)}_{s},H^{(j)}_{s}]$ for $i\neq j$ are uniformly
integrable martingales. Taking the expectation, we get
\begin{align*}
&\E\left| Y_{t}^{n}\right|^{2}+\int_{t}^{T}\|Z_{s}^{n}\|^{2}ds\\
&\leq\left| \xi\right|
^{2}+2\E\int_{t}^{T}Y_{s^-}^{n}f(s,Y_{s^-}^{n},Z_{s}^{n}) ds
+2\E\int_{t}^{T}Y_{s^-}^{n}\phi(s,Y_{s^-}^{n})dA_s\nonumber\\
&+\E\int_{t}^{T}|g(s,Y_{s^-}^{n})|^{2}ds +2\E\int_{t}^{T} Y_{s^-}^n dK_{s}^{n},
\end{align*}
where we have used
$\displaystyle{\int_{t}^{T}(Y^{n}_{s^-}-S_{s})dK^{n}_{s}\leq 0}$ and
the fact that
\begin{eqnarray*}
\int_{t}^{T}Y^{n}_{s}dK^{n}_{s}&=&
\int_{t}^{T}(Y^{n}_{s^-}-S_{s})dK^{n}_{s}+\int_{t}^{T}S_{s}dK^{n}_{s}\leq
\int_{t}^{T}S_{s}dK^{n}_{s}.
\end{eqnarray*}
Using $(\textbf{H2})$ and the elementary inequality $2ab\leq \gamma
a^{2}+\frac{1}{\gamma} b^2,\ \forall\gamma>0$,
\begin{eqnarray*}
2Y_{s}^{n}f(s,Y_{s}^{n},Z_{s}^{n})&\leq&(c\gamma_1+\frac{1}{\gamma_1})|Y^{n}_{s}|^{2}
+2c\gamma_1\|Z^{n}_{s}\|^{2}+2\gamma_1 f_s^{2},\\
2Y_{s}^{n}\phi(s,Y_{s}^{n})&\leq& (\gamma_2-2|\beta|)|Y^{n}_{s}|^{2}+\frac{1}{\gamma_2}\phi_s^{2},\\
|g(s,Y_{s}^{n})|^{2}
&\leq&2c|Y^{n}_{s}|^{2}+2g_s^{2}.
\end{eqnarray*}
Taking expectation in both sides of the inequality (\ref{b2}) and
choosing $\displaystyle\gamma_1=\frac{1}{4c}$,
$\displaystyle\gamma_2=|\beta|$, we obtain for all
$\varepsilon >0$
\begin{align}\label{a2}
&
\E\left|Y_{t}^{n}\right|^{2}+|\beta|\E\int_{t}^{T}\left|Y_{s}^{n}\right|^{2}dA_s
+\frac{1}{2}\E\int_{t}^{T}\left\|Z_{s}^{n}\right\|^{2}ds \nonumber\\
&\leq
C\E\left\{|\xi|^{2}+\int^{t}_{0}|Y^{n}_{s}|^{2}ds+\int^{t}_{0}f_s^{2}ds+
\int^{t}_{0}\phi_s^{2}dA_s+\int^{t}_{0}g_s^{2}ds\right\}\nonumber\\
& +\frac{1}{\varepsilon}\E\left(\sup_{0\leq s\leq
t}(S_{s}^+)^2\right)+\varepsilon\E \left(K_{T}^{n}-K_t^n\right)^2.
\end{align}
On the other hand, we get from (\ref{h2}) that for all $0 \leq t
\leq T$,
\begin{equation}\label{K:n}
K_t^n=Y_t^n-\xi-\int_0^t f(s,Y_{s^-}^n,Z_s^n) ds -\int_0^t
\phi(s,Y_{s^-}^n) dA_s-\int_0^t g(s,Y_{s^-}^n) dB_s+\sum_{i=1}^{m}\int_0^t (Z_s^n)^{(i)}dH^{(i)}_s.
\end{equation}
So by used standard computations, we get
\begin{align}\label{kn1}
\E(K^n_T-K^{n}_t)^2\leq C\E\left\{|\xi|^{2}+\int^{t}_{0}f_s^{2}ds+
\int^{t}_{0}\phi_s^{2}dA_s+\int^{t}_{0}g_s^{2}ds+\int_0^t
\left|Y_s^n\right|^2ds\right.\nonumber\\
\left.+\E\left(\sup_{0\leq s\leq
t}(S_{s}^+)^2\right)+\int_0^t
\left|Y_s^n\right|^2dA_s+\int_0^t\|Z_s^n\|^2ds\right\}.
\end{align}
Substituting Equation $(\ref{kn1})$ to Equation $(\ref{a2})$ and choosing $\varepsilon$ small enough such that $\varepsilon C<\min(1/2,|\beta|)$,
yields
\begin{align*}
&
\E\left\{|Y_{t}^{n}|^{2}+\int_{t}^{T}\left|Y_{s}^{n}\right|^{2}dA_s
+\int_{t}^{T}\|Z_{s}^{n}\|^{2}ds+|K_{T}^{n}|^{2}\right\}\\
&\leq
C\E\left\{|\xi|^{2}+\int^{T}_{0}f_s^{2}ds+\int^{T}_{0}\phi_s^{2}dA_s
+\int^{T}_{0}g_s^{2}ds+\sup_{0\leq t\leq T}(S_{t}^{+})^{2}\right\}.
\end{align*}
From this, Gronwall's inequality and the Burkholder-Davis-Gundy
inequality \cite{DM}, we get
\begin{eqnarray*}
\E\left\{\sup_{0\leq t\leq T}
|Y_{t}^{n}|^{2}+\int_{t}^{T}\|Z_{s}^{n}\|^{2}ds+|K_{T}^{n}|^{2}\right\}
&\leq & C\E\left\{|\xi|^{2}+\int^{T}_{0}f_s^{2}ds+\int^{T}_{0}\phi_s^{2}dA_s\right.\nonumber\\
&&+\left.\int^{T}_{0}g_s^{2}ds+\sup_{0\leq t\leq
T}(S_{t}^{+})^{2}\right\},
\end{eqnarray*}
which end the proof of this Lemma.
\end{proof}

\noindent \begin{proof}[Proof of Theorem \ref{Th:exis-uniq}] {\bf
Existence} The proof of existence will be divided in two steps.

\noindent \textit{Step}\ 1.  $g $ does not dependent on $\left(
Y,Z\right)$.
More precisely, we consider the following
equation
\begin{eqnarray}
Y_{t}&=&\xi+\int_{t}^{T}f(s,Y_{s^-},Z_s)ds+
\int_{t}^{T}\phi(s,Y_{s^-})dA_s+\int_{t}^{T}g(s)\,dB_{s}\nonumber\\
&&-\sum^{m}_{i=1}\int_{t}^{T}Z^{(i)}_{s}dH^{(i)}_{s}+K_T-K_t, \; 0\leq t\leq T.\label{h4}
\end{eqnarray}
The penalized equation is given by
\begin{eqnarray}
Y_{t}^{n}&=&\xi+\int_{t}^{T}f(s,Y_{s^-}^n,Z_s^n)ds+n\int_{t}^{T}(Y_{s^-}^{n}-S_{s})^{-}ds+
\int_{t}^{T}\phi(s,Y_{s^-}^{n})dA_s\nonumber\\
&&+\int_{t}^{T}g(s)\,dB_{s}-\sum^{m}_{i=1}\int_{t}^{T}(Z_{s}^{n})^{(i)}_{s}dH^{(i)}_{s}, \; 0\leq t\leq T.
\label{h3}
\end{eqnarray}
\noindent Since the sequence of functions $(y\mapsto
n(y-S_{t})^{-})_{n\geq 1}$ is nondecreasing, then  thanks to the
comparison theorem \ref{Thm:comp}, the sequence $\left( Y^{n}\right)
_{n>0}$ is non-decreasing. Hence, Lemma \ref{lem:1} implies that
there exists a $\mathcal{F}_t$- progressively measurable process $Y$
such that $Y_{t}^{n}\nearrow Y_{t}$ $a.s$.
\noindent Recall that $Y_{t}^{n}\nearrow Y_{t}$ $a.s$. Then,
Fatou's lemma and Lemma \ref{lem:1} ensure
\begin{eqnarray*}
\E\left( \sup_{0\leq t\leq T }\left| Y_{t}\right| ^{2}\right)
<+\infty,
\end{eqnarray*}
It then follows from Lemma \ref{lem:1} and Lebegue's dominated
convergence theorem  that
\begin{eqnarray}
\E\left( \int_{0}^{T }\left| Y_{s}^{n}-Y_{s}\right| ^{2} ds
\right)\longrightarrow 0,\,\,\ \mbox{as}\,\  n\rightarrow
\infty.\label{b9}
\end{eqnarray}
Next, for\ $n \geq p \geq 1$, by It\^{o}'s formula and together with assumptions
$({\bf H2})$, yields
\begin{eqnarray*}
&&\E\left\{\left| Y_{t}^{n}-Y_{t}^{p}\right| ^{2}+
\int_{t}^{T}\left| Y_{s}^{n}-Y_{s}^{p}\right| ^{2}dA_s +
\int_{t}^{T}\left\|
Z_{s}^{n}-Z_{s}^{p}\right\| ^{2}ds\right\} \\
&\leq &
C\E\left\{\int_{t}^{T}|Y_{s}^{n}-Y_{s}^{p}|^{2}ds+\sup_{0\leq s\leq
T}\left(Y_{s}^{n}-S_{s}\right)^{-}K_{T}^{p} +\sup_{0\leq s\leq
T}\left( Y_{s}^{p}-S_{s}\right) ^{-} K_{T}^{n}\right\},
\end{eqnarray*}
which, by Gronwall lemma, H\"{o}lder inequality and Lemma \ref{lem:1}
respectively, implies
\begin{eqnarray}
\E\left\{\left| Y_{t}^{n}-Y_{t}^{p}\right| ^{2} + \int_{t}^{T}\left\|
Z_{s}^{n}-Z_{s}^{p}\right\| ^{2}ds\right\} &\leq &
C\left\{\E\left(\sup_{0\leq s\leq T}|\left( Y_{s}^{n}-S_{s}\right)
^{-}|^{2}\right)\right\}^{1/2}\nonumber\\
&&+C\left\{ \E\left(\sup_{0\leq s\leq T}|\left(
Y_{s}^{p}-S_{s}\right) ^{-}|^{2}\right)\right\}^{1/2}. \label{b11'}
\end{eqnarray}
Let us admit for the moment the following result.
\begin{lemma}\label{lem:2}
If $g $ does not dependent on $\left( Y,Z\right)$, then for
each $n\in\N^{*}$,
\begin{eqnarray*}
\E\left(\sup_{0\leq t\leq T }\left|\left( Y_{t}^{n}-S_{t}\right)
^{-}\right| ^{2}\right)\longrightarrow 0,\,\,\ \mbox{as}\,\ n
\longrightarrow \infty.
\end{eqnarray*}
\end{lemma}
We can now conclude. Indeed, it
follows from Lemma \ref{lem:2} that,
\begin{eqnarray*}
\E\left\{\left| Y_{s}^{n}-Y_{s}^{p}\right| ^{2} +\int_{t}^{T}\left\|
Z_{s}^{n}-Z_{s}^{p}\right\| ^{2}ds\right\}\longrightarrow 0,\,\,\
\mbox{as}\,\ n,p \longrightarrow \infty. \label{h1}
\end{eqnarray*}
Finally, from Burkh\"{o}lder-Davis-Gundy's inequality, we obtain
\begin{eqnarray*}
\E\left( \sup_{0\leq s\leq T}\left| Y_{s}^{n}-Y_{s}^{p}\right|
^{2}+\int_{t}^{T}\left\| Z_{s}^{n}-Z_{s}^{p}\right\| ^{2}ds\right)
\longrightarrow 0,\,\, \mbox{ as }\,  n,p\longrightarrow
\infty,\label{b12}
\end{eqnarray*}
and from (\ref{K:n}) we can deduce
\begin{eqnarray*}
\E\left\{\sup_{0\leq s\leq T}\left| K_{s}^{n}-K_{s}^{p}\right|
^{2}\right) \longrightarrow 0,\mbox{ as }n,p\rightarrow \infty ,
\label{b13}
\end{eqnarray*}
which provides that the sequence of processes $(Y^{n},Z^{n},K^{n})$
is Cauchy in the Banach space $\mathcal{E}^{2,m}$. Consequently, there exists
a triplet $(Y,Z,K)\in\mathcal{E}^{2,m}$ such that
\begin{eqnarray*}
\E\left\{\sup_{0\leq s\leq T}\left| Y_{s}^{n}-Y_{s}^{{}}\right|
^{2}+\int_{t}^{T}\left\|Z^{n}_{s}-Z_{s}\right\|^{2}ds+\sup_{0\leq
s\leq T}\left| K_{s}^{n}-K_{s}^{{}}\right| ^{2}\right) \rightarrow
0,\mbox{ as }n\rightarrow \infty .
\end{eqnarray*}
It remains to show that $(Y,Z,K)$ solves the reflected GBDSDE driven by L\'{e}vy processes $(\xi,
f,\phi,g,S)$. In this fact, since $(Y^{n}_t,K^{n}_t)_{0\leq t\leq T}$ tends to $(Y_t,K_t)_{0\leq t\leq T}$
uniformly in t in probability, the measure $dK^n$ converges to $dK$ weakly in probability, so that $\int_{t}^{T}(Y_{s^-}^{n}-S_{s})dK_{s}^{n}\rightarrow \int_{t}^{T}(Y_{s^-}-S_{s})dK_{s}$ in probability as $n\rightarrow \infty$. Obviously, $\int_{t}^{T}(Y_{s^-}-S_{s})dK_{s}\geq 0$,
while, on the other hand, for all $n\geq 0$, $\int_{t}^{T}(Y_{s^-}^{n}-S_{s})dK_{s}^{n}\leq 0$.

Hence
\begin{eqnarray*}
\int_{t}^{T}(Y_{s^-}-S_{s})dK_{s}= 0,\;\; a.s
\end{eqnarray*}
 Finally, passing to the limit
in $(\ref{h3})$ we proved that $(Y,Z,K)$ verifies (\ref{h4}) and is the solution of the reflected GBDSDE
$(\xi,f, g, S)$ driven by the L\'{e}vy processes. We finally return to the proof of Lemma 2.2.

\begin{proof}[Proof of Lemma \ref{lem:2}]
Since $Y^{n}_{t}\geq Y^{0}_{t}$, we can w.l.o.g. replace $S_{t}$ by
$S_{t}\vee Y^{0}_{t}$, i.e. we may assume that $\E(\sup_{0\leq
t\leq T}S_{t}^{2})<\infty$. We want to compare a.s. $Y_t$ and $S_t$
for all $t\in[0,T]$. In this, let us introduce the following processes
$$
\left\{
\begin{array}{ll}
&\displaystyle \overline{\xi}:=\xi+\int_{t}^{T}g\left(
s\right)  dB_{s}\\
& \displaystyle\overline{S}_{t}:=S_{t}+\int_{t}^{T}g\left(
s\right)  dB_{s}\\
& \displaystyle\overline{Y}_{t}^{n}:=Y_{t}^{n}+\int_{t} ^{T}g\left(
s\right)  dB_{s}.
\end{array}
\right.
$$
Hence,  \begin{equation}\label{Ynbar}
\overline{Y}_{t}^{n}=\overline{\xi}+\int_{t}^{T}f\left(
s,Y_{s^-}^n,Z_s^n\right) ds+n\int_{t}^{T}\left(\overline{Y}_{s^-}^{n}
-\overline{S_{s} }\right) ^{-}ds+\int_{t}^{T}\phi\left(s,
Y_{s^-}^n\right) dA_s-\sum_{i=1}^{m}\int_{t}^{T}(Z_{s}^{n})^{(i)}dH^{(i)}_{s}.
\end{equation}
and we define \quad
$\overline Y_t:=\displaystyle\sup_n \overline Y^n_t$.\\
From Theorem \ref{Thm:comp}, we have that a.s.,
$\overline{Y}^{n}_t\geq\widetilde{Y}_{t}^{n},\, 0\leq t\leq T, \,
n\in\N^{*},$ where $\left\{
(\widetilde{Y_{t}}^{n},\widetilde{Z}_{t}^{n}),\mbox{ }0\leq t\leq
T\right\} $ is the unique solution of the  GBDSDE driven by the L\'{e}vy processes
\begin{eqnarray*}
\widetilde{Y}_{t}^{n} &=&\overline{S}_{T}+\int_{t}^{T}f\left(
s,Y_{s^-}^n,Z_s^n\right)
ds+n\int_{t}^{T}(\overline{S}_{s}-\widetilde{Y}_{s^-}^{n})ds
+\int_{t}^{T}\phi\left(s, Y_{s^-}^n\right)
dA_s-\sum_{i=1}^{m}\int_{t}^{T}(\widetilde{Z} _{s}^{n})^{(i)}dH^{(i)}_{s}.
\end{eqnarray*}
Now let ${\bf G}=(\mathcal{G}_{t})_{0\leq t\leq T}$ be a filtration defined by $\mathcal{G}_{t}=\mathcal{F}^{L}_{t}\vee\mathcal{F}^{B}_{T}$ and $\nu $
a ${\bf G}$-stopping time such that $0\leq \nu \leq T$. Then,
applying Itô formula to $\widetilde{Y}^{n}_t e^{-n(t-\nu)}$, we have
\begin{eqnarray*}
\widetilde{Y}_{\nu }^{n} &=&\E\left\{ e^{-n(T-\nu)}\overline{S}_{T}
+\int^{ T}_{\nu }e^{-n(\nu-s)}f(s,Y_{s^-}^{n},Z_{s}^{n})ds+n\int^{ T}_{\nu }e^{-n(\nu-s)}\overline{S}_{s}ds \right.\nonumber \\
&&\left.+\int^{ T}_{\nu }e^{-n(\nu-s)}\phi(s,Y_{s^-}^{n})dA_s\mid
{\cal G} _{\nu }\right\}. \label{c'2}
\end{eqnarray*}
It is easily seen that
\begin{eqnarray*}
e^{-n(T-\nu)}\overline{S}_{T}+n\int_{\nu}^{T}e^{-n(s-\nu)}\overline{S}_{s}ds\rightarrow \overline{S}_{\nu
}{\bf 1}_{\{\nu <T\}}+\overline{S}_{T}{\bf 1}_{\{\nu =T\}}\,\, a.s.,\, \mbox{and in}\, L^{2}(\Omega)\, \mbox{as}\,\, n\rightarrow \infty,\label{S}
\end{eqnarray*}
and the conditional expectation converges also
in $L^{2}(\Omega)$. Moreover, we get
 \begin{eqnarray*}
\E\left(\int_{\nu}^{T}e^{-n(s-\nu)}f(s,Y_{s^-}^{n},Z_{s}^{n})ds\right)^{2}&\leq&
 \frac{C}{2n}\E\left( \int_{0 }^{T}
(f_s^{2}+|Y_{s}^{n}|^{2}+\|Z_{s}^{n}\|^{2}) ds\right),
\end{eqnarray*}
and \begin{eqnarray*}
\E\left(\int_{\nu}^{T}e^{-n(s-\nu)}\phi(s,Y_{s}^{n})dA_s\right)^{2}
&\leq& \E\left[|A_T|\left(\int_{0
}^{T}(\phi_s^{2}+K^{2}|Y_{s}^{n}|^{2}) dA_s\right)\right]<C,
\end{eqnarray*}
which provide
\begin{eqnarray*}
\E\left(\int_{\nu}^{T}e^{-n(\nu-s)}f(s,Y_{s^-}^{n},Z_{s}^{n})ds+\int_{\nu}^{T}e^{-n(s-\nu)}\phi(s,Y_{s}^{n})dA_s|\mathcal{G}_{\nu}\right)\longrightarrow
0 \label{ds}
\end{eqnarray*}
in $L^{2}(\Omega)$ as $n\rightarrow \infty$.

Consequently,
\begin{eqnarray*}
\widetilde{Y}_{\nu }^{n}\longrightarrow \overline{S}_{\nu
}{\bf 1}_{\{\nu <T\}}+\overline{S}_{T}{\bf 1}_{\{\nu=T\}}\;\; \mbox{in}\, L^{2}(\Omega), \; \mbox{as}\; n\rightarrow \infty.
\end{eqnarray*}

Therefore $Y_{\nu}\geq S_{\nu}$ a.s.
From this and the section theorem \cite{DM}, we deduce that $Y_{t}\geq S_{t}$ for all
$t\in[0,T]$ and then
\begin{eqnarray*}(Y_{t}^{n}-S_{t})^{-}\searrow 0,\;\; 0\leq t\leq T,\;\; a.s.
\end{eqnarray*}
Since $(Y_{t}^{n}-S_{t})^{-}\leq (S_{t}-Y_{t}^{0})^{+}\leq
\left| S_{t}\right| +\left| Y_{t}^{0}\right|$and the result follows from
the dominated convergence theorem.
\end{proof}

\noindent \textit{Step}\ 2. The general case. In light of the above
step, and for any $(\bar{Y},\bar{Z})\in \mathcal{S}^{2}(\R)\times \mathcal{M}^{2}(\R^m) $, the reflected GBDSDE driven by L\'{e}vy processes
\begin{equation*}
Y_{t}=\xi+\int^{T}_t f(s,{Y}_s,{Z}_s)ds
+\int_{t}^{T}\phi(s,Y_s)dA_s+\int_{t}^{T}g(s,\bar{Y}_s)\,dB_{s}
-\sum^{m}_{i=1}\int_{t}^{T}Z_{s}^{(i)}dH^{(i)}_{s}+K_T-K_t
\end{equation*}
has a unique solution $(Y,Z,K)$. So, we can define the mapping
$$
\begin{array}{lrlll}
\Psi:&\mathcal{S}^{2}(\R)\times \mathcal{M}^{2}(\R^m)&\longrightarrow&\mathcal{S}^{2}(\R)\times \mathcal{M}^{2}(\R^m)\\
&(\bar{Y},\bar{Z})&\longmapsto&(Y,Z)=\Psi(\bar{Y},\bar{Z}).
\end{array}
$$
Now, let $(Y,Z),\ (Y',Z')$ in $\mathcal{S}^{2}(\R)\times \mathcal{M}^{2}(\R^m)$ and $(\bar{Y},\bar{Z}),
(\bar{Y'},\bar{Z'})$ in $\mathcal{S}^{2}(\R)\times \mathcal{M}^{2}(\R^m)$ such that
$(Y,Z)=\Psi(\bar{Y},\bar{Z})$ and $(Y',Z')=\Psi(\bar{Y'},\bar{Z'})$.
Putting $\Delta \eta=\eta-\eta'$ for any process $\eta$, and by
virtue of It\^o's formula, we have
\begin{eqnarray*}
&&\E e^{-\mu t}|\Delta Y_t|^2+\E\int_{t}^T e^{-\mu s}\|\Delta Z_s\|^2ds\\
&&=2\E\int_{t}^T e^{-\mu s}\Delta
Y_s\left\{f(s,{Y}_{s^-},{Z}_{s})-f(s,{Y'}_{s^-},{Z'}_s)\right\}ds
+2\E\int_{t}^T e^{-\mu s} \Delta Y_{s}\left\{\phi(s,{Y}_{s^-})-\phi(s,{Y'}_{s^-})\right\}dA_s\\
 &&+2\E\int_{t}^T e^{-\mu s} \Delta Y_s d(\Delta K_s)+\int_{t}^Te^{-\mu
s}\left|g(s,\bar{Y}_{s^-})-g(s,\bar{Y'}_{s^-})\right|^2ds-\mu\E\int_{t}^T
e^{-\mu s} \left|\Delta Y_s\right|^2 ds.
\end{eqnarray*}
But since $\displaystyle\E\int_{t}^T e^{-\mu s} \Delta Y_s d(K_s-K'_s)\leq
0$, then from $({\bf H2})$ there exists constant $\gamma$ such that,
\begin{eqnarray*}
&&(\mu-\gamma)\E\int_{t}^T
e^{-\mu s} \left|\Delta Y_s\right|^2 ds+\frac{1}{2}\E\int_{t}^T e^{-\mu s}\|\Delta Z_s\|^2ds\\
&&\leq c\E\left(\int_{t}^T e^{-\mu s}\left|\Delta \bar {Y}_s\right|^2ds\right)
\end{eqnarray*}

Now choose $\mu=\gamma+2c$ and define $\bar{c}=2c$, we
obtain
\begin{eqnarray*}
&&\bar{c}\E\int_0^t
e^{-\mu s} \left|\Delta Y_s\right|^2 ds+\frac{1}{2}\E\int_0^t e^{-\mu s}\|\Delta Z_s\|^2ds\\
&&\leq \frac{1}{2}\left(\bar{c}\E\int_0^t e^{-\mu
s}\left|\Delta \bar {Y}_s\right|^2ds+\frac{1}{2}\E\int_0^t e^{-\mu
s}\left|\Delta \bar {Z}_s\right|^2ds\right).
\end{eqnarray*}
Consequently, $\Psi$ is a strict contraction on $\mathcal{S}^{2}(\R)\times \mathcal{M}^{2}(\R^m)$ equipped with the norm 
\begin{eqnarray*}
\|Y,Z)\|^{2}=\bar{c}\E\int_0^t
e^{-\mu s} \left|Y_s\right|^2 ds+\frac{1}{2}\E\int_0^t e^{-\mu s}\|Z_s\|^2ds
\end{eqnarray*}
and it has a unique fixed point, which is the unique solution our BDSDE.

{\bf Uniqueness} Assume $\left( Y_{t},Z_{t},K_{t}\right)_{0\leq t\leq T}$ and  $(Y_{t}^{\prime },Z_{t}^{\prime },K_{t}^{\prime })_{0\leq t\leq T}$ are two solutions of the reflected GBDSDE $(\xi,f,g,\phi,S)$ driven by L\'{e}vy processes. Set $\Delta Y_{t}=Y_{t}-Y_{t}^{\prime},\, \Delta Z_{t}=Z_{t}-Z_{t}^{\prime }$ and
$\Delta K_{t}=K_{t}-K_{t}^{\prime }$. Applying Itô's formula to $(\Delta Y)^2$ on the interval $[t,T]$ and taking
expectation on both sides, it follows that
\begin{eqnarray*}
&&\E\left| \Delta Y_{t}\right| ^{2}+\E\int_{t}^{T}\|\Delta Z_{s}\| ^{2}ds \\
&=& 2\E\int_{t}^{T}\Delta
Y_{s}(f(s,Y_{s^-},Z_{s}^{{}})-f(s,Y_{s^-}^{\prime },Z_{s}^{\prime
}))ds+2\E\int_{t}^{T}|g(s,Y_{s^-})-g(s,Y_{s^-}^{\prime})|^{2}ds\nonumber\\
&&+2\E\int_{t}^{T}\Delta Y_{s}(\phi(s,Y_{s^-})-\phi(s,Y_{s^-}^{\prime
}))dA_s+2\E\int_{t}^{T}\Delta Y_{s}d(\Delta K_s)\\
&\leq& 4c^{2}\E\int_{t}^{T}|\Delta Y_{s}|^{2}ds+\frac{1}{4c^{2}}\E\int_{t}^{T}|f(s,Y_{s^-},Z_{s}^{{}})-f(s,Y_{s^-}^{\prime },Z_{s}^{\prime})|^{2}ds\\
&&+\beta \E\int_{t}^{T}|\Delta Y_{s}|^{2}ds+c\E\int_{t}^{T}|\Delta Y_{s}|^{2}ds\\
&\leq&4c^{2}\E\int_{t}^{T}|\Delta Y_{s}|^{2}ds+\frac{2c^{2}}{4c^{2}}\E\int_{t}^{T}|\Delta Y_{s}|^2ds+\frac{2c^{2}}{4c^{2}}\E\int_{t}^{T}\|\Delta Z_{s}\|^2ds\\
&&+\beta \E\int_{t}^{T}|\Delta Y_{s}|^{2}ds+c\E\int_{t}^{T}|\Delta Y_{s}|^{2}ds\\
&\leq&(4c^{2}+c+\frac{1}{2})\E\int_{t}^{T}|\Delta Y_{s}|^{2}ds+\frac{1}{2}\E\int_{t}^{T}\|\Delta Z_{s}\|^2ds,
\end{eqnarray*}
here we have used the assumption $({\bf H2})$, the inequality $2ab\leq \frac{a^2}{\gamma}+\gamma b^2\;  (\forall\ \gamma>0)$ and the fact that
\begin{eqnarray*}
\int^{T}_{0}\Delta Y_{s} d(\Delta K_s) \leq 0. \label{D6}
\end{eqnarray*}
So, we have

\begin{eqnarray*}
\E|\Delta Y_{t}|^{2} &\leq &(4c^{2}+c+\frac{1}{2})\E\int_{t}^{T}|\Delta Y_{s}|^2ds.
\end{eqnarray*}
Henceforth, from Gronwall's inequality, it follows that $\E|\Delta Y_{t}|^{2}=\E|Y_{t}-Y^{\prime}_t|^{2}=0, 0\leq t\leq T$, that is, $Y_{t}=Y^{\prime}_t$ a.s.
Then, we also have $\E\int_{t}^{T}\|\Delta Z_{s}\|^2ds=\E\int_{t}^{T}\|Z_{s}-Z^{\prime}\|^2ds=0$ and $Z_t=Z^{\prime}_t,\; K_t=K^{\prime}_t$ follows.
The proof is complete now.
\end{proof}

\section {Connection to reflected stochastic PDIEs with nonlinear
Neumann boundary condition} \setcounter{theorem}{0}\setcounter{equation}{0}
 In this section, we study the link between reflected GBDSDEs driven by L\'{e}vy processes
and the solution of a class of reflected stochastic PDIEs with a nonlinear Neumann
boundary condition. Suppose that our L\'{e}vy processes $L$ has bounded
jump and has the following L\'{e}vy decomposition:
\begin{eqnarray*}
L_t=bt+\int_{|z|\leq 1}z(N_t(.,dz)-t\nu(dz))
\end{eqnarray*}
where $N_t(\omega, dz)$ denotes the random measure such that $\int_{\Lambda}N_t(.,dz)$ is a Poisson process
with parameter $\nu(\Lambda)$ for all set $\Lambda\;  (0\notin\Lambda)$.

Let $\Theta = (-\theta,\theta)$ and $e : [-\theta, \theta] \rightarrow \R$ such that $e(-\theta) = 1$ and $e(\theta) = -1$.
Consider the following reflected SDE:
\begin{eqnarray}
X_t=x+\int^{T}_t\sigma(X_{s^-})dL_s+\eta_t, \label{RSDEJ1}
\end{eqnarray}
and
\begin{eqnarray}
\eta_t=\int^{t}_0 e(X_{s})d|\eta|_s, \; \mbox{with}\; |\eta|_t=\int^{t}_0{\bf 1}_{\{X_s\in\partial\Theta\}}d|\eta|_s.\label{RSDEJ2}
\end{eqnarray}
Under adequate conditions (see \cite{Ot} or \cite{MR}), there exists a unique pair of
progressively measurable processes $(X,\eta)$ that satisfies $(\ref{RSDEJ1})$ and $(\ref{RSDEJ2})$, and for
any progressively measurable process $V$ which is right continuous having left-hand limits and take its values in $\bar{\Theta}$, we have
\begin{eqnarray*}
\int^{T}_0 (X_{s}-V_s)d|\eta|_s\geq 0.
\end{eqnarray*}
In order to attain our main result in this section, we give a Lemma appeared in \cite{NS2}.
\begin{lemma}
let $c:\Omega\times[0,T]\times\R\rightarrow\R$ be a measurable function such that
\begin{eqnarray*}
    |c(s,y)|\leq a_s(y^2\wedge|y|)\;\; a.s.,
\end{eqnarray*}
where $\{a_s, s\in [0, T]\}$ is a non-negative predictable process such that $E\int^T_0 a^2_sds <\infty$. Then, for each $0\leq t\leq T$, we have
\begin{eqnarray*}
\sum_{t\leq s\leq T}c(s,\Delta L_s)=\sum^{m}_{i=1}\int^{T}_{t}\langle c(s,.),p_i\rangle_{L^2(\nu)}dH^{(i)}_s+\int^{T}_{t}\int_{\R}c(s,y)d\nu(y)ds
\end{eqnarray*}
\end{lemma}

Let $l:\R\rightarrow\R,\; h:[0,T]\times\R\rightarrow\R$ be continuous functions such that
\begin{description}
\item $(i)\; \E\left(|l(X_{T})|^{2}+\sup_{0\leq t\leq T}|h(t,X_t)|^{2}\right)<\infty,$
\item $(ii)\; l(x)\geq h(T,x), \; $ for all $x\in\R$.
\end{description}

Next, consider the following reflected GBDSDE:
\begin{eqnarray}
Y_{t}&=&l(X_T)+\int_{t}^{T}f(s,X_{s^-}Y_{s^-},Z_{s})ds+\int_{t
}^{T}\phi(s,X_{s^-},Y_{s^-})d|\eta|_s+\int_{t}^{T}g(s,X_{s^-},Y_{s^-})\,dB_{s}\nonumber\\
&&-\sum_{i=1}^{\infty}\int_{t}^{T}Z^{(i)}_{s}dH^{(i)}_{s}+K_{T}-K_t,\,\ 0\leq t\leq
T,\label{GBDSDEmarkovian}
\end{eqnarray}
such that the following holds $\P$-a.s
\begin{description}
\item  $(i)$ $Y_{t}\geq h(t,X_{t}),\,\,\,\ 0\leq t\leq T$,\,\
\item  $(ii)$\ $\displaystyle \int_{0}^{ T}\left( Y_{t^-}-h(t,X_{t})\right)
dK_{t}=0$.
\end{description}

Define
\begin{center}
$
u^{1}(t,x,y)=u(t,x+y)-u(t,x)-\frac{\partial u}{\partial x}(t,x)y,
$
\end{center}
where $u$ is the solution of the following reflected stochastic PDIE with a nonlinear
Neumann boundary condition:\\
\begin{eqnarray}
\left\{
\begin{array}{l}
\displaystyle \min\left\{u(t,x)-h(t,x),\; \frac{\partial u}{\partial t}(t,x)+a'\sigma(x)
\frac{\partial u}{\partial x}(t,x)+f(t,x,u(t,x),(u^{i}(t,x))_{i}^{m})\right.\\\\
\left.\,\,\,\,\,\,\,\,\,\,\,\,\,\,\,\,\,\ +\int_{\R}u^1(t,x,y)d\nu(y)+
g(t,x,u(t,x))dB_{t}\right\}=0,\,\,\
(t,x)\in[0,T]\times\Theta\\\\
\displaystyle e(x)\frac{\partial u}{\partial
x}(t,x)+\phi(t,x,u(t,x))=0,\,\,\ (t,x)\in[0,T]\times\{-\theta,\theta\},
\\\\ u(T,x)=l(x),\,\,\,\,\,\,\ x\in\Theta,
\end{array}\right.
\label{RSPDIE}
\end{eqnarray}
where $a'=a+\int_{\{|y|\geq 1\}}y\nu(dy),\; dB_t=\dot{B}_t$ denotes a white noise and
\begin{flushleft}
$u^{(1)}(t,x)=\int_{\R}u^1(t, x, y)p_1(y)\nu(dy) +\frac{\partial u}{\partial x}(t, x)(\int_{\R} y^2\nu(dy))^{1/2}$
\end{flushleft}
and for $2\leq i\leq m,\; u^{(i)}(t,x)=\int_{\R}u^1(t, x, y)p_i(y)\nu(dy)$.

Suppose that $u$ is $\mathcal{C}^{1,2}$ function such that $\frac{\partial u}{\partial t}$ and $\frac{\partial^2 u}{\partial x^2}$
is bounded by polynomial function of $x$, uniformly in $t$. Then we have the following
\begin{theorem}
The unique adapted solution of $(\ref{GBDSDEmarkovian})$ is given by
\begin{eqnarray*}
Y_t &=& u(t,X_t),\\
Z^{(1)}_t &=&\int_{\R}u^1(t,X_{t^-}, y)p_1(y)\nu(dy) + \frac{\partial u}{\partial x} \sigma(X_{t^-})\left(\int_{\R}y^2\nu(dy)\right)^{1/2}\\
Z^{(i)}_t &=&\int_{\R} u^{1}(t,X_{t^-},y)p_i(y)\nu(dy), \; 2\leq i\leq m,
\end{eqnarray*}
\end{theorem}
\begin{proof}
For each $n\geq 1$, let
$\{^{n}Y_{s},{}^{n}Z_{s},\,\ 0\leq s\leq T\}$ denote the
solution of the GBDSDE
\begin{eqnarray*}
^{n}Y_{s}&=&l(X_{T})+\int^{T}_{s}f(r,X_{r^-},{}^{n}Y_{r^-},
{}^{n}Z_{r})dr+n\int^{T}_{s}({}^{n}Y_{r^-}-h(r,X_{r}))^{-}dr\nonumber\\
&&+\int^{T}_{s}\phi(r,X_{r^-},^{n}Y_{r^-})d|\eta|_r
\int^{T}_{s}g(r,X_{r^-},^{n}Y_{r^-})dB_{r}-\sum^{m}_{i=1}\int^{T}_{s}
{}^{n}Z^{(i)}_{r}dH^{(i)}_{r}. \label{c4}
\end{eqnarray*}
It is know from Hu and Yong \cite{HY1} that
\begin{eqnarray*}
{}^{n}Y_t &=& u_n(t,X_t),\\
{}^{n}Z^{(1)}_t &=&\int_{\R}u_n^1(t,X_{t^-}, y)p_1(y)\nu(dy) + \frac{\partial u_n}{\partial x} \sigma(X_{t^-})\left(\int_{\R}y^2\nu(dy)\right)^{1/2}\\
{}^{n}Z^{(i)}_t &=&\int_{\R} u_n^{1}(t,X_{t^-},y)p_i(y)\nu(dy), \; 2\leq i\leq m,
\end{eqnarray*}
where $u_n$ is the classical solution of stochastic PDIE:
\begin{eqnarray}
\left\{
\begin{array}{l}
\displaystyle \frac{\partial u_n}{\partial t}(t,x)+a'\sigma(x)
\frac{\partial u_n}{\partial x}(t,x)+f_n(t,x,u_n(t,x),(u_n^{i}(t,x))_{i}^{m})\\\\
\,\,\,\,\,\,\,\,\,\,\,\,\,\,\,\,\,\ +\int_{\R}u_n^1(t,x,y)d\nu(y)+
g(t,x,u_n(t,x))dB_{t}=0,\,\,\
(t,x)\in[0,T]\times\Theta\\\\
\displaystyle e(x)\frac{\partial u_n}{\partial
x}(t,x)+\phi(t,x,u_n(t,x))=0,\,\,\ (t,x)\in[0,T]\times\{-\theta,\theta\},
\\\\ u_n(T,x)=l(x),\,\,\,\,\,\,\ x\in\Theta,
\end{array}\right.
\label{SPDIE}
\end{eqnarray}
where $f_{n}(t,x,y,z)=f(t,x,y,z)+n(y-h(t,x))^{-}$.

Applying It\^{o}'s formula to $u_n(s,X_s)$, we obtain
\begin{eqnarray}
u_n(T,X_T)-u_n(t,X_t)&=&\int_{t}^{T}\frac{\partial u_n}{\partial s}(s,X_{s^-})ds+\int_{t}^{T}e(X_s)\frac{\partial u_n}{\partial x}(s,X_{s})d|\eta|_s\nonumber\\
&&+\int_{t}^{T}\sigma(X_{s^-})\frac{\partial u_n}{\partial x}(s,X_{s^-})dL_s\nonumber\\
&&+\sum_{t\leq s\leq T}[u_n(s,X_s)-u_n(s,X_{s^-})
-\frac{\partial u_n}{\partial x}(s,X_{s^-})\Delta X_s].\label{Ito}
\end{eqnarray}
Lemma 4.1 applied to $u_n(s,X_{s^-}+y)-u_n(s,X_{s^-})-\frac{\partial u_n}{\partial x}(s,X_{s^-})y$ shows
\begin{eqnarray}
\sum_{t\leq s\leq T}[u_n(s,X_s)-u_n(s,X_{s^-})-\frac{\partial u_n}{\partial x}(s,X_{s^-})\Delta X_s]
&=&\sum_{i=1}^{m}\int_{t}^{T}\left(\int_{\R}u_n^{1}(s,X_{s^-},y)p_{i}(y)\nu(dy)\right)dH^{(i)}\nonumber\\
&&+\int_{t}^{T}\left(\int_{\R}u_n^{1}(s,X_{s^-},y)\nu(dy)\right)ds.\label{Applema}
\end{eqnarray}
Note that
\begin{eqnarray}
L_t=Y^{(1)}_t+t\E L_1=\left(\int_{\R}y^2\nu(dy)\right)^{1/2}H^{(1)}+t\E L_1, \label{Levyproperty}
\end{eqnarray}
where $\E L_1=a+\int_{\{|y|\geq 1\}}y\nu(dy)$.
Hence, substituting $(\ref{RSDEJ2})$, $(\ref{Applema})$ and $(\ref{Levyproperty})$ into $(\ref{Ito})$ together with $(\ref{SPDIE})$ yields
\begin{eqnarray*}
&&l(X_T)-u_n(t,X_t)\nonumber\\
&=&\int_{t}^{T}\left[\frac{\partial u_n}{\partial s}(s,X_{s^-})+(a+\int_{|y|\geq 1}y\nu(dy))\sigma(X_{s^-})\frac{\partial u_n}{\partial x}(s,X_{s^-})+\int_{\R}u_n^{1}(s,X_{s^-},y)\nu(dy)\right]ds\nonumber\\
&&+\int_{t}^{T}e(X_s)\frac{\partial u_n}{\partial x}(s,X_{s}){\bf 1}_{\{X_s\in\partial\Theta\}}d|\eta|_s\nonumber\\
&&+\int_{t}^{T}\left[\int_{\R}u_n^{1}(s,X_{s^-},y)p_{1}(y)\nu(dy)+\sigma(X_{s^-})\frac{\partial u_n}{\partial x}(s,X_{s^-})\left(\int_{\R}y^2\nu(dy)\right)^{1/2}\right]dH^{(1)}_s\nonumber\\
&&+\sum_{i=2}^{m}\int_{t}^{T}\left(\int_{\R}u_n^{1}(s,X_{s^-},y)p_{i}(y)\nu(dy)\right)dH^{(i)}_s.\nonumber\\
&=&-\int_{t}^{T}f(s,X_{s^-},u_n(s,X_s),(u_n(s,X_s))_{i=1}^{m})ds+n\int_{t}^{T}(u_n(s,X_s)-h(s,X_s))^{-}ds\\
&&-\int_{t}^{T}g(s,X_{s^-},u_n(s,X_s))dB_s-\int_{t}^{T}\varphi(s,X_{s^-},u_n(s,X_s))d|\eta|_s\\
&&+\int_{t}^{T}\left[\int_{\R}u_n^{1}(s,X_{s^-},y)p_{1}(y)\nu(dy)+\sigma(X_{s^-})\frac{\partial u_n}{\partial x}(s,X_{s^-})\left(\int_{\R}y^2\nu(dy)\right)^{1/2}\right]dH^{(1)}_s\nonumber\\
&&+\sum_{i=2}^{m}\int_{t}^{T}\left(\int_{\R}u_n^{1}(s,X_{s^-},y)p_{i}(y)\nu(dy)\right)dH^{(i)}_s.
\end{eqnarray*}
From which passing in the limit on $n$, and using the previous section we get the desired result of the Theorem.
\end{proof}

Next, we give a example of reflected stochastic PDIEs with a nonlinear Neumann
boundary condition.
\begin{example}
Suppose the L\'{e}vy process $L$ has the form of $L_t = at+\sum_{i=1}^{\infty}(N^{(i)}-\alpha_i t)$, where $(N^{(i)})^{\infty}_{i=0}$ is a sequence of independent Poisson
processes with parameters $(\alpha^{i})^{\infty}_{i=0}, (\alpha_i>0)$. Its L\'{e}vy measure is $\nu(dx) =\sum_{i=1}^{\infty}\alpha_i\delta_{\beta_i}(dx)$, where $\delta_{\beta_i}$ denotes the positive point mass measure at $\beta_i\in\R$ of size $1$. Furthermore, we assume that $\sum_{i=1}^{\infty}\alpha_i|\beta_i|^2<\infty$. Recall that this L\'{e}vy process has only one jumps size and no continuous parts so that $H^{(1)}_t=\sum_{i=1}^{\infty}\frac{\beta_i}{\sqrt{\alpha_i}}(N_t^{(i)}-\alpha_{i}t)$ and $H^{(i)}_t=0, i\geq 2$ (see \cite{NS2}). Let $(Y,Z,K)$ be the unique solution of the following reflected GBDSDEs
\begin{eqnarray*}
Y_{t}&=&l(X_T)+\int_{t}^{T}f(s,X_{s^-}Y_{s^-},Z_{s})ds+\int_{t
}^{T}\phi(s,X_{s^-},Y_{s^-})d|\eta|_s+\int_{t}^{T}g(s,X_{s^-},Y_{s^-})\,dB_{s}\nonumber\\
&&-\sum_{i=1}^{\infty}\int_{t}^{T}Z^{(i)}_{s}d(N_s^{(i)}-\alpha_{i}s)+K_{T}-K_t,\,\ 0\leq t\leq
T\label{GBDSDEmarkovian}
\end{eqnarray*}
such that the following holds $\P$-a.s
\begin{description}
\item  $(i)$ $Y_{t}\geq h(t,X_{t}),\,\,\,\ 0\leq t\leq T$,\,\
\item  $(ii)$\ $\displaystyle \int_{0}^{ T}\left( Y_{t^-}-h(t,X_{t})\right)
dK_{t}=0$.
\end{description}
Then
\begin{eqnarray*}
Y_t &=& u(t,X_t),\\
Z^{(1)}_t &=&\alpha_1u^1(t,X_{t^-}, \beta_1)p_1(\beta_1) + \sigma(X_{t^-})\frac{\partial u}{\partial x}(t,X_{t^-})\left(\sum_{i=1}^{\infty}\alpha_i|\beta_i|^{2}\right)^{1/2} \\
Z^{(i)}_t &=&\alpha_iu^{1}(t,X_{t^-},\beta_i)p_i(\beta), \; i\geq 2,
\end{eqnarray*}
where $u$ is the solution of the following reflected stochastic PDIEs with a nonlinear
Neumann boundary condition:
\begin{eqnarray*}
\left\{
\begin{array}{l}
\displaystyle \min\left\{u(t,x)-h(t,x),\; \frac{\partial u}{\partial t}(t,x)+a'\sigma(x)
\frac{\partial u}{\partial x}(t,x)+f(t,x,u(t,x),\frac{\partial u}{\partial x}(t,x))\right.\\\\
\left.\,\,\,\,\,\,\,\,\,\,\,\,\,\,\,\,\,\ +\sum_{i=1}^{\infty}\alpha_iu^1(t,x,\beta_i)+
g(t,x,u(t,x))dB_{t}\right\}=0,\,\,\
(t,x)\in[0,T]\times\Theta\\\\
\displaystyle e(x)\frac{\partial u}{\partial
x}(t,x)+\phi(t,x,u(t,x))=0,\,\,\ (t,x)\in[0,T]\times\{-\theta,\theta\},
\\\\ u(T,x)=l(x),\,\,\,\,\,\,\ x\in\Theta.
\end{array}\right.
\end{eqnarray*}
\end{example}

\label{lastpage-01}
\end{document}